\documentclass[12pt]{amsart}
\usepackage{amssymb}
\usepackage{mathrsfs}
\newtheorem{theorem}{Theorem}

\newtheorem{definition}[theorem]{Definition}
\newtheorem{remark}[theorem]{Remark}
\newtheorem{corollary}[theorem]{Corollary}

\newcommand{\M}{\mbox{${\mathbb{M}}$}}
\newcommand{\A}{\mbox{${\mathcal A}$}}
\newcommand{\Br}{\mbox{${\mathcal B}$}}
\newcommand{\Q}{\mbox{${\mathcal{QM}}$}}
\begin{document}
\title[The unit ball of an injective operator space has an extreme point]{The unit ball of an injective operator space has an extreme point}
\author{Masayoshi~Kaneda}
\address{Department of Mathematics and Natural Sciences, College of Arts and Sciences, American University of Kuwait, P.O. Box 3323, Safat 13034 Kuwait}
\email{mkaneda@uci.edu}
\date{\today}
\begin{abstract}We define an $AW^*$-TRO as an off-diagonal corner of an $AW^*$-algebra, and show that the unit ball of an $AW^*$-TRO has an extreme point. In particular, the unit ball of an injective operator space has an extreme point, which answers a question raised in \cite{Kaneda 2016} affirmatively. We also show that an $AW^*$-TRO (respectively, an injective operator space) has an ideal decomposition, that is, it can be decomposed into the direct sum of a left ideal, a right ideal, and a two-sided ideal in an $AW^*$-algebra (respectively, an injective $C^*$-algebra). In particular, we observe that $AW^*$-TRO, hence an injective operator space, has an algebrization which admits a quasi-identity.
\end{abstract}
\maketitle
Recall that an operator space $X$ is called a \emph{triple system} or a \emph{ternary ring of operators} (\emph{TRO} for short) if there exists a complete isometry $\iota$ from $X$ into a $C^*$-algebra such that $\iota(x)\iota(y)^*\iota(z)\in\iota(X)$ for all $x,y,z\in X$. A theorem of Ruan and Hamana (independently) states that an operator space $X$ is injective if and only if it is an off-diagonal corner of an injective $C^*$-algebra, i.e., there exist an injective $C^*$-algebra $\A$ and projections $p,q\in\A$ (meaning $p=p^2=p^*$ and $q=q^2=q^*$) such that $X$ is completely isometric to $p\A q$ (Theorem~4.5 in \cite{Ruan 1989} and Theorem~3.2~(i) in \cite{Hamana 1999}). In particular, an injective operator space is a TRO. Noting that an injective $C^*$-algebra is monotone complete and hence an $AW^*$-algebra, the Ruan-Hamana theorem motivates the following definition. (The reader is referred to \cite{Saito & Wright 2015} for a modern account of and recent progress in monotone complete $C^*$-algebras and $AW^*$-algebras.)
\begin{definition}\label{de:AW-TRO}We say that an operator space $X$ is an \textbf{\em $AW^*$-TRO} if there exist an $AW^*$-algebra $\A$ and projections $p,q\in\A$ such that $X$ is completely isometric to $p\A q$.
\end{definition}
\begin{remark}{\em
\begin{enumerate}
\item Our definition of an $AW^*$-TRO is weaker than the one given in \cite{Pluta 2013}~(Definition~6.2.1) where an $AW^*$-TRO is defined as a TRO whose linking $C^*$-algebra is an $AW^*$-algebra. This condition is so strong that even some injective operator spaces fail to be $AW^*$-TROs in this sense. For instance, a countably-infinite-dimensional column Hilbert space is an injective operator space (\cite{Robertson 1991}) and hence a TRO, however, its linking $C^*$-algebra is not unital, and so is not an $AW^*$-algebra. In our belief, disqualifying an injective operator space, which is an off-diagonal corner of an $AW^*$-algebra, from being an $AW^*$-TRO is not befitting to its name, so in this paper we use the term ``$AW^*$-TRO'' in the sense of Definition~\ref{de:AW-TRO} above, and hence an injective operator space is an $AW^*$-TRO. Also this definition is consistent with that of a $W^*$-TRO which is an off-diagonal corner of a $W^*$-algebra but its linking $C^*$algebra need not be a $W^*$-algebra.
\item With this modified definition of $AW^*$-algebras and $\mathscr{L}_T:=\begin{bmatrix}p\A p&p\A q\\q\A p&q\A q\end{bmatrix}\subseteq\M_2(\A)$, where $\A$, $p$, and $q$ are as in Definition~\ref{de:AW-TRO}, all Theorems, Corollaries, and Lemmas in Sections~6.2 and 6.3 of \cite{Pluta 2013} are valid except for Statement~6.2.2 and Corollary~6.2.6 there.
\end{enumerate}}
\end{remark}
\begin{theorem}\label{th:extpt}The unit ball (always assumed to be norm-closed) of an $AW^*$-TRO has an extreme point. In particular, the unit ball of an injective operator space has an extreme point, which answers a question raised in \cite{Kaneda 2016} (Question~2) affirmatively.
\end{theorem}
\begin{proof}Let $X$ be an $AW^*$-TRO. We may assume that $X=p\A q$, where $\A$ is an $AW^*$-algebra and $p,q\in\A$ are projections. By the comparison theorem in \cite{Heunen & Reyes 2013}, there exist unique central projections $r,t,l\in\A$ satisfying $r+t+l=1$ such that $rp\prec rq$, $tp\sim tq$, and $lp\succ lq$. (Here $rp\prec rq$ means $rp\preceq rq$ but $rp\nsim rq$, however, $0\prec0$ is allowed.) That is, there exist partial isometries $u,v,w\in\A$ such that $uu^*=rp$, $u^*u\le rq$, $vv^*=tp$, $v^*v=tq$, $ww^*\le lp$, and $w^*w=lq$. Let $e:=u+v+w\,(\in p\A q)$, then it is easy to check that $(p-ee^*)\A(q-e^*e)=\{0\}$. Thus by a variation of Kadison's theorem (Theorem~1 in \cite{Kadison 1951}; see Proposition~1.4.8 in \cite{Pedersen 1979} or Proposition~1.6.5 in \cite{Sakai 1971} for the variation we need here), $e$ is an extreme point of the unit ball of $p\A q$.
\end{proof}From the proof above we obtain ``ideal decompositions'' for $AW^*$-TROs and injective operator spaces similar to the ones done for TROs with predual in \cite{Kaneda 2013}. The technique we use here is to embed an off-diagonal corner into the diagonal corners which is a modification of the technique developed in \cite{Blecher & Kaneda 2004} and is employed in \cite{Kaneda 2013}.
\begin{corollary}\label{co:ideal decomposition}An $AW^*$-TRO (respectively, injective operator space) can be decomposed into the direct sum of TROs $X_T$, $X_L$, and $X_R$:$$X=X_T\stackrel{\infty}{\oplus}X_L\stackrel{\infty}{\oplus}X_R$$ so that there is a complete isometry $\iota$ from $X$ into an $AW^*$-algebra (respectively, an injective $C^*$-algebra) in which $\iota(X_T)$, $\iota(X_L)$, and $\iota(X_R)$ are a two-sided, left, and right ideal, respectively, and$$\iota(X)=\iota(X_T)\stackrel{\infty}{\oplus}\iota(X_L)\stackrel{\infty}{\oplus}\iota(X_R).$$
\end{corollary}
\begin{proof}Let $X$ be an $AW^*$-TRO, and assume that $X=p\A q$, where $\A$ is an $AW^*$-algebra and $p,q\in\A$ are projections. Let $r,t,l\in p\A q$ as in the proof of Theorem~\ref{th:extpt}, and put $X_T:=tX$, $X_L:=lX$, and $X_R:=rX$, then $X=X_T\stackrel{\infty}{\oplus}X_L\stackrel{\infty}{\oplus}X_R$. Let $\Br:=p\A p\stackrel{\infty}{\oplus}q\A q$ which is an $AW^*$-algebra, since $p\A p$ and $q\A q$ are so by Theorem~2.4 in \cite{Kaplansky 1951}. For each $x\in X$, let $x_T:=tx$, $x_L:=lx$, and $x_R:=rx$, and define a mapping $\iota:X\to\Br$ by $\iota(x):=(x_T+x_L)e^*\oplus e^*x_R$, then clearly $\iota(X)=\iota(X_T)\stackrel{\infty}{\oplus}\iota(X_L)\stackrel{\infty}{\oplus}\iota(X_R)$. We claim that $\iota$ is a complete isometry. $\|\iota(x)\|=\max\{\|(x_T+x_L)e^*\|,\|e^*x_R\|\}=\max\{\|(x_T+x_L)e^*e(x_T+x_L)^*\|^{1/2},\|x_R^*ee^*x_R\|^{1/2}\}
=\max\{\|x_Tv^*vx_T^*+x_Lw^*wx_L^*\|^{1/2},\|x_R^*uu^*x_R\|^{1/2}\}=\max\{\|xtx^*+xlx^*\|^{1/2},\|x^*rx\|^{1/2}\}
=\max\{\|(t+l)x\|,\|rx\|\}=\|(t+l+r)x\|=\|x\|$, which shows that $\iota$ is an isometry. A similar calculation works at the matrix level, which concludes that $\iota$ is a complete isometry. Clearly, $\iota(X_T)$, $\iota(X_L)$, and $\iota(X_R)$ are respectively a two-sided, left, and right ideals in $\Br$, and thus we are done. The proof in the case that $X$ is an injective operator space is identical noting that $\Br$ is an injective $C^*$-algebra in this case.
\end{proof}
\begin{remark}{\em
\begin{enumerate}
  \item In the proof above it is also possible to define $\iota:X\to\Br$ by $\iota(x):=x_Le^*\oplus e^*(x_R+x_T)$ for $x\in X$.
  \item A TRO $X$ with predual can be considered as an off-diagonal corner of a von Neumann algebra $\A$ (see the beginning of the proof of the Theorem in  \cite{Kaneda 2013}), thus the above argument gives an alternate and simpler proof of the Theorem in \cite{Kaneda 2013} noting that $\Br$ in the above proof is a von Neumann algebra in this case. The simplicity of this alternate proof is attributed to the use of the comparison theorem for projections.
\end{enumerate}}
\end{remark}The following corollary is straightforward from the corollary above. The reader is referred to \cite{Kaneda 2003}, \cite{Kaneda 2007}, or \cite{Kaneda & Paulsen 2004} for quasi-multipliers and algebrizatioins of operator spaces, and Definition~4.2~(i) in \cite{Kaneda 2016} for quasi-identities.
\begin{corollary}An $AW^*$-TRO, hence an injective operator space, has an algebrization which admits a quasi-identity of norm $1$.
\end{corollary}
\begin{proof}It is straightforward to check that $(v+w)e^*\oplus e^*u$ serves as a quasi-identity of $\iota(X)$ in the proof of Corollary~\ref{co:ideal decomposition}.
\end{proof}
\begin{remark}An element $e^*\in X^*$ in the proof of Corollary~\ref{co:ideal decomposition} can be identified as a quasi-multiplier of $X$ noting that $X^*\subseteq\Q(X)$ if $X$ is a TRO, where $\Q(X)$ is the quasi-multiplier space of $X$, and $\iota$ is the ``algebrization'' by $e^*$.
\end{remark}
\section*{Acknowledgment}The author wishes to express his gratitude to Professor~Kazuyuki~Sait\^{o} for finding and fixing a gap in the initial draft of the proof of Theorem~\ref{th:extpt}.

\end{document}